\documentclass[a4paper,11pt]{amsart}
\usepackage{amsmath}
\usepackage{mathrsfs}
\usepackage{amsfonts}
\usepackage{graphicx}
\usepackage{color}
\usepackage{amsfonts}
\usepackage{amssymb}
\usepackage[hidelinks]{hyperref}
\usepackage{todonotes}
\usepackage{esint}

\usepackage{palatino}

\usepackage[abbrev,backrefs]{amsrefs}

\usepackage{mathtools}

\newtheorem{theorem}{Theorem}[section]

\newtheorem{definition}[theorem]{Definition}

\newtheorem{lemma}[theorem]{Lemma}

\newtheorem{proposition}[theorem]{Proposition}
\newtheorem{remark}[theorem]{Remark}

\numberwithin{equation}{section}

\newcommand{\interior}[1]{%
	{\kern0pt#1}^{\mathrm{o}}%
}

\def\esssup{\mathop{\rm ess\,sup\,}}

\def\supp{{\rm supp\,}}

\newcommand{\R}{\mathbb{R}}

\newcommand{\barint}{
	\rule[.036in]{.12in}{.009in}\kern-.16in \displaystyle\int }

\newcommand{\barcal}{\mbox{$ \rule[.036in]{.11in}{.007in}\kern-.128in\int $}}

\usepackage{enumitem}
\makeatletter
\let\@wraptoccontribs\wraptoccontribs
\makeatother

\mathchardef\mhyphen="2D

\title{Estimates for the wave equation on $\beta$-dimensional spaces of measures}
\author[R. Basak]{Riju Basak}
\address[R. Basak]{Department of Mathematics, National Taiwan Normal University, No. 88, Section 4, Tingzhou Road, Wenshan District, Taipei City, Taiwan 116, R.O.C.
}
\email{rijubasak52@ntnu.edu.tw}

\author[D. Spector]{Daniel Spector}
\address[D. Spector]{Department of Mathematics, National Taiwan Normal University, No. 88, Section 4, Tingzhou Road, Wenshan District, Taipei City, Taiwan 116, R.O.C.
\newline
National Center for Theoretical Sciences\\No. 1 Sec. 4 Roosevelt Rd., National Taiwan
University\\Taipei, 106, Taiwan
\newline
Department of Mathematics, University of Pittsburgh, Pittsburgh, PA 15261 USA
}
\email{spectda@protonmail.com}

\begin{document}
	\begin{abstract}
In this paper, we establish Miyachi-Peral-type fixed-time estimates for wave multipliers acting on $\beta$-dimension stable spaces of measures.  Our estimates give a refinement of known estimates for the Hardy space.  From these bounds, we deduce corresponding estimates for the wave equation with measure data.
\end{abstract}
	\maketitle
\section{Introduction}
\subsection{Main Results}
For $b\in \mathbb{R}$ and $f\in \mathcal{S}(\mathbb{R}^n)$, define the Fourier multiplier operator 
\begin{align}\label{main_multiplier}
    T_{b}f(x)=\int_{\mathbb{R}^n} (1+4\pi^2|\xi|^2)^{-b/2}e^{i|\xi|} \widehat{f}(\xi) e^{2\pi i \xi\cdot x} \, d\xi,
\end{align}
where 
\begin{align*}
 \widehat{f}(\xi)=\int_{\mathbb{R}^n} f(x) e^{-2\pi i\xi\cdot x} \, dx
\end{align*}
is our convention for the Fourier transform.  Further define
\begin{align*}
b_p=\frac{n+1}{2}-\frac{1}{p}.
\end{align*}

It is a classical result of Miyachi \cite{Miyachi-singular} (see also  \cite{Miyachi-wave,Peral-80}) that for suitable values of $b$, $T_{b}$ admits an extension as a bounded operator from the Hardy space $\mathcal{H}^1(\mathbb{R}^n)$ to $L^p(\mathbb{R}^n)$ and $L^1(\mathbb{R}^n)$ to $L^p(\mathbb{R}^n)$.  For the Hardy space $\mathcal{H}^1(\mathbb{R}^n)$ one has estimates up to the endpoint $b=b_p$:
\begin{theorem}\label{M_Hardy}
Let $p \in [1,\infty]$ and suppose $b \geq b_p$.  There exists a constant $C=C(n,p)$ such that 
    \begin{align*}
        \|T_{b}f\|_{L^p(\mathbb{R}^n)} \leq C \|f\|_{\mathcal{H}^1(\mathbb{R}^n)} 
    \end{align*}
    for all $f \in \mathcal{H}^1(\mathbb{R}^n)$.  
\end{theorem}

For the space $L^1(\mathbb{R}^n)$ one has estimates, provided $b > b_p$:
\begin{theorem}\label{M_L1}
Let $p \in [1,\infty]$ and suppose $b > b_p$.  There exists a constant $C=C(n,p)$ such that 
    \begin{align*}
        \|T_{b}f\|_{L^p(\mathbb{R}^n)} \leq C \|f\|_{L^1(\mathbb{R}^n)} 
    \end{align*}
    for all $f \in L^1(\mathbb{R}^n)$.  
\end{theorem}

\begin{remark}\label{miyachis-multiplier}
More precisely, in \cite[Theorem 4.2 I-i,iii on p.~284]{Miyachi-singular} Miyachi establishes analogous inequalities for the multiplier operator
\begin{align}\label{main-operator}
    \widetilde{T}_{b}f(x)=\int_{\mathbb{R}^n} \psi(\xi)(2\pi|\xi|)^{-b}e^{i|\xi|} \widehat{f}(\xi) e^{2\pi i \xi\cdot x} \, d\xi,
\end{align}
where $\psi \in C^\infty(\mathbb{R}^n)$ satisfies $\psi \equiv 0$ on $B(0,1)$ and $\psi \equiv 1$ outside $B(0,2)$.   Theorems \ref{M_Hardy} and \ref{M_L1} follow from these results and standard convolution estimates, see the discussion below in Section \ref{preliminaries}.
\end{remark}

The results in \cite{Miyachi-wave,Miyachi-singular,Peral-80} are a part of the broader body of work on estimates for oscillatory Fourier multipliers, including for solutions of the wave equation on compact manifolds \cite{Seeger-Sogge-Stein,CFS}, the Heisenberg group \cite{Muller-Stein,Muller-Seeger}, solvable Lie groups \cite{Muller-Thiele, WangYan}, Grushin operators \cite{JT-Pisa}, Hermite operators \cite{Narayanan-Thangavelu,BDHH-IMRN}, twisted Laplacians \cite{Basak-J},  as well as other oscillatory multiplier estimates \cite{BCGW-Revista,Bui-Anconca-Duong,BDN-Revista-2020,MMN} and related maximal operators \cite{Rogers-Villaroya,Cho-Lee-Li,Kinoshita-Ko-Shiraki,Ko-Lee-Shiraki}.  The study of such estimates are of interest in their own right, though have a special relevance in physics because of their application to the wave equation, a point we return to in the sequel.

It is interesting to notice the discrepancy between the assumption $b \geq b_p$ for functions in the Hardy space $\mathcal{H}^1(\mathbb{R}^n)$ and $b > b_p$ for functions in $L^1(\mathbb{R}^n)$, in particular in light of the recent results concerning Riesz potentials acting on constrained subspaces of $L^1(\mathbb{R}^n;\mathbb{R}^k)$ \cite{ASW,GRVS, HRS,HS, HP, RSS, SSVS, Spector, Spector-Stolyarov, Stolyarov, Stolyarov-1,VanSchaftingen_2013 }.  A scalar structure underlying this phenomenon has been introduced in \cite{Spector-Stolyarov} in terms of the spaces of $\beta$-dimension stable measures $DS_\beta(\mathbb{R}^n)$, $\beta \in (0,n]$.  The spaces $DS_\beta(\mathbb{R}^n)$ are intermediate between the Hardy space and the space of finite Radon measures:  For $\beta \in (0,n)$,
\begin{align}
   \mathcal{H}^1(\mathbb{R}^n) = DS_n(\mathbb{R}^n)\subsetneq DS_\beta(\mathbb{R}^n) \subsetneq M_b(\mathbb{R}^n).
\end{align} 
They support Sobolev inequalities, e.g. 
 \begin{align}\label{sobolev}
        \|I_\alpha \mu\|_{L^{n/(n-\alpha)}(\mathbb{R}^n)} \leq C \|\mu\|_{DS_\beta(\mathbb{R}^n)},
    \end{align}
for all $\mu \in DS_\beta(\mathbb{R}^n)$, and also Besov-Lorentz and trace improvements, see \cite[Theorem A, B on p.~4]{Spector-Stolyarov}).
Here
\begin{align*}
I_\alpha \mu(x) := \frac{1}{\gamma(\alpha)} \int_{\mathbb{R}^n} \frac{d\mu(y)}{|x-y|^{n-\alpha}}
\end{align*}
is the Riesz potential of order $\alpha \in (0,n)$ of a Radon measure $\mu$ for a suitable normalization constant $\gamma(\alpha)$ (see, e.g. \cite[p.~117]{S}).  
The inequalities \eqref{sobolev} can be seen as refinements of a classical inequality of Stein and Weiss \cite[Theorem G on p.~60]{SteinWeissHp}:
\begin{align}
        \|I_\alpha f\|_{L^{n/(n-\alpha)}(\mathbb{R}^n)} \leq C \|f\|_{\mathcal{H}^1(\mathbb{R}^n)}, 
\end{align}
for all $f \in \mathcal{H}^1(\mathbb{R}^n)$, while the lower bound of the Hausdorff dimension established in \cite[Theorem H on p.~6]{Spector-Stolyarov}:
\begin{align*}
\operatorname*{dim_\mathcal{H}} \mu := \sup \left\{ \gamma \in (0,n] \colon \mathcal{H}^\gamma(E)=0 \implies |\mu|(E)=0 \right\} \geq \beta
\end{align*}
for all $\mu \in DS_\beta(\mathbb{R}^n)$ is analogous to Stein and Weiss's result for the Hardy space \cite[Theorem E on p.~53]{SteinWeissHp}.
We refer the reader to Section \ref{preliminaries} below for the definition of the spaces $DS_\beta(\mathbb{R}^n)$.

The sharpening of results for the Riesz potentials in the inequality \eqref{sobolev} suggests the possibility of corresponding results for multipliers or more general convolution operators.  In this paper we obtain the first such estimates outside the elliptic setting.  First, when $p\geq 2$, we have the following result up to the endpoint $b=b_p$ with no further restrictions on $\beta \in (0,n]$.
\begin{theorem}\label{Thm-p-bigger-2}
Let $p\in [2, \infty)$, and suppose $b \geq b_p$.  For $\beta\in (0, n]$, there exists a constant $C=C(n,p, \beta)$ such that 
    \begin{align*}
        \|T_{b}\mu\|_{L^p(\mathbb{R}^n)} \leq C \|\mu\|_{DS_{\beta}(\mathbb{R}^n)} 
    \end{align*}
    for all $\mu \in DS_{\beta}(\mathbb{R}^n)$.  
\end{theorem}
When $1\leq p<2$, the cancellation present in the operator is more subtle and we impose further assumptions.  Our main result is

\begin{theorem}\label{Thm-p-less-2}
    Let $n\geq 3$, $p\in [1, 2)$, and $b\geq b_p$.  For $\beta \in (n-1,n]$, there exists a constant $C=C(n,p, \beta)$ such that 
    \begin{align*}
        \|T_{b}\mu\|_{L^p(\mathbb{R}^n)} \leq C \|\mu\|_{DS_{\beta}(\mathbb{R}^n)} 
    \end{align*}
    for all $\mu \in DS_{\beta}(\mathbb{R}^n)$.  
\end{theorem}

When $n= 2$ we have the following result.

\begin{theorem}\label{Thm-p-less-2-2}
    Let $n=2$, $p\in (1, 2)$, and $b\geq b_p$.  For $\beta \in (\frac{2}{p},2]$, there exists a constant $C=C(\beta,p)$ such that 
    \begin{align*}
      \|T_{b}\mu\|_{L^p(\mathbb{R}^2)} \leq C \|\mu\|_{DS_{\beta}(\mathbb{R}^2)} 
   \end{align*}
   for all $\mu \in DS_{\beta}(\mathbb{R}^2)$.  
\end{theorem}

One observes that in contrast to Theorem \ref{Thm-p-bigger-2}, Theorems \ref{Thm-p-less-2} and \ref{Thm-p-less-2-2} are only valid for the dimension stable spaces above a threshold.  This is likely an artifact of our proof:  For small cubes we adapt an argument of Miyachi, where the use of an $L^2$ estimate for a fractional operator on atoms imposes (for $n\geq 3$) the more general condition $\beta>(n-1)/2+1/p$, see Lemma \ref{lemma-small-cubes} below; for large cubes, our use of the decay of the Littlewood-Paley modes of the kernel in the proof of Lemma \ref{LP-arguments} is ultimately what gives the lower bound $n-1$ (when $n\geq 3$ and is less restrictive than the argument for small cubes when $n=2$).  
It would be interesting to know the optimal value of $\beta \in (0,n]$ in these latter two theorems, whether one has an estimate for $p=1$ when $n=2$ for some value of $\beta \in (0,2)$, and also the validity of Theorem \ref{Thm-p-bigger-2} for $p=\infty$ for any value of $\beta\in (0,n)$.

\subsection{An Application to the Wave Equation with Measure Data}
An easy scaling argument shows that with the same restrictions on $\beta$ as given by our results, the dilated operator
\begin{align*}
    T^t_{b}f(x)=\int_{\mathbb{R}^n} (1+4\pi^2|t\xi|^2)^{-b/2}e^{it|\xi|} \widehat{f}(\xi) e^{2\pi i \xi\cdot x} \, d\xi
\end{align*}
admits the estimate
\begin{align*}
        \|T^t_{b}\mu\|_{L^p(\mathbb{R}^n)} \lesssim t^{-n/p'}\|\mu\|_{DS_{\beta}(\mathbb{R}^n)} 
\end{align*}
for all $\mu \in DS_{\beta}(\mathbb{R}^n)$. This of course connects with one of the origins of the study of such multipliers, that their boundedness properties can be useful in the deduction of results for solutions of the wave equation.  Following Peral \cite{Peral-80}, Seeger, Sogge, and Stein \cite[Theorem 4.1]{Seeger-Sogge-Stein},  and M\"uller and Seeger \cite[Equation (2) on p.~1052]{Muller-Seeger}, we consider an application to the Cauchy problem.  Let $n\geq 3$ and $\beta \in (n-1,n]$.  Suppose $(I-t^2\Delta)^{b_p/2}f, t(I-t^2\Delta)^{(b_p-1)/2}g \in DS_\beta(\mathbb{R}^n)$, and consider the problem of finding $u: \mathbb{R}^n \times \mathbb{R}^+ \to \mathbb{R}$ which satisfies
\begin{align*}
u_{tt} - \Delta u &= 0, \quad (x,t) \in \mathbb{R}^n \times \mathbb{R}^+,\\
u(x,0) &= f,\quad x \in \mathbb{R}^n,\\
u_t(\cdot,0)&=g, \quad x \in \mathbb{R}^n.
\end{align*}

One can check that
\begin{align*}
u(x,t) &= \int_{\mathbb{R}^n} \left(\cos(2\pi t|\xi|)\widehat{f}(\xi)+\frac{\sin(2\pi t|\xi|)}{2\pi |\xi|}\widehat{g}(\xi)\right)e^{2\pi i \xi\cdot x}\;d\xi,
\end{align*}
satisfies the equation and boundary condition in the sense of distributions, while the results obtained in this paper imply that for any $p \in [1,\infty)$ one has the estimate
\begin{align}
\|  u(\cdot,t)\|_{L^p(\mathbb{R}^n)} \lesssim \frac{1}{t^{n/p'}}\left(\|(I-t^2\Delta)^{b_p/2}f\|_{DS_\beta}+ \|t(I-t^2\Delta)^{(b_p-1)/2}g\|_{DS_\beta}\right).
\end{align}
Note that our estimate includes the case $p=1$, and in particular when $n=3$ for $f=0$ one obtains
\begin{align}
\|  u(\cdot,t)\|_{L^1(\mathbb{R}^n)} \lesssim t \|g\|_{DS_\beta}.
\end{align}

The plan of the paper is as follows.  In Section \ref{preliminaries} we recall several definitions, kernel representation and estimates for oscillatory multipliers.  
We also establish some preliminary estimates that will be useful in the sequel.  
In Section \ref{p_larger_than_2} we give the proof of Theorem \ref{Thm-p-bigger-2}.  In Section \ref{p_small_than_2} we provide the proofs of Theorems  \ref{Thm-p-less-2} and \ref{Thm-p-less-2-2}.

\section{Preliminaries}\label{preliminaries}
Throughout the paper we use $C$ to denote a constant which depends only on parameters and the dimension, but not the functions estimated, and which may change from line to line.  We also use the notation $A\lesssim B$ to mean that $A\leq CB$ for a constant $C>0$ which depends only on the parameters.  For a measurable set $A$ such that $0<|A|<+\infty$ and a measurable function $g:\mathbb{R}^n\to \mathbb{R}$ which is integrable over $A$, we write
\begin{align*}
    \fint_A g(y)\;dy \equiv  \left(g\right)_{A} := \frac{1}{|A|} \int_A g(y)\;dy
\end{align*}
to denote the average value of $g$ over $A$. The two shorthand notations will be used interchangeably according to convenience of display.

The theorems in this paper are asserted for $b\geq b_p$, though it suffices to obtain such estimates for the critical value $b=b_p$.  In particular, for $b>b_p$, the estimates follow from the embedding $DS_\beta(\mathbb{R}^n) \hookrightarrow L^1(\mathbb{R}^n)$ and Theorem \ref{M_L1}, and therefore we restrict our attention to the critical value.  When $1 \leq p<2$, we focus on estimates for the operator \eqref{main-operator}.  The sufficiency of these estimates follows from the decomposition
\begin{align}\label{Decomposition-operator}
T_{b}f = \overline{T}_{b}f+ \nu_b\ast \widetilde{T}_{b}f
\end{align}
where
\begin{align}\label{goodkernel}
    \overline{T}_bf(x)=\int_{\mathbb{R}^n} (1-\psi(\xi))(2\pi |\xi|)^{-b}e^{i|\xi|} \widehat{f}(\xi) e^{2\pi i \xi\cdot x} \, d\xi \equiv \int_{\mathbb{R}^n} \overline{K}_b(x-y)f(y)\;dy,
\end{align}
with $\psi \in C^\infty(\mathbb{R}^n)$ as in Remark \ref{miyachis-multiplier} and
\begin{align*}
 \widehat{\nu}_b(\xi) = \frac{(2\pi|\xi|)^{b}}{(1+4\pi^2|\xi|^2)^{b/2}}.
\end{align*}
\noindent
Indeed, a simple computation shows the kernel $\overline{K}_b$ in \eqref{goodkernel} satisfies $\overline{K}_b \in L^1(\mathbb{R}^n)\cap L^\infty(\mathbb{R}^n)$ and is therefore bounded map from $L^1(\mathbb{R}^n)$ to $L^p(\mathbb{R}^n)$ for any $1\leq p \leq +\infty$, while \cite[Lemma 2 on p.~133]{S} shows $\nu_b$ is a bounded measure and therefore the corresponding operator defined by convolution with $\nu_b$ maps $L^p(\mathbb{R}^n)$ to $L^p(\mathbb{R}^n)$ for any $1\leq p \leq +\infty$.

We require several results concerning the operator $\widetilde{T}_b$ defined in \eqref{main-operator}.  By \cite{Miyachi-wave}*{Proposition 2} we can write the operator $\widetilde{T}_{b}$ as a convolution operator $\widetilde{T}_{b}f(x)=K_{b}\ast f(x)$ with the following estimate of the kernel $K_{b}$.
\begin{proposition}\label{wave-ker-esti}
Let $b\in \mathbb{R}$. Then 
\begin{enumerate}
    \item for any $\alpha\in \mathbb{N}^n$ there exists a constant $C_\alpha$ such that 
\begin{align*}
    \left|\partial^{\alpha}K_{ b}(x) \right| \leq C_\alpha \left| 1- |x| \right|^{b-\frac{n+1}{2}-|\alpha|} \quad \text{as} \quad |x|\rightarrow 1.
\end{align*}
\item  for any $\alpha\in \mathbb{N}^n$  and $N \in \mathbb{N}$ there exists a constant $C_{\alpha,N}$ such that 
\begin{align*}
    \left|\partial^{\alpha}K_{ b}(x) \right| \leq C_{\alpha,N}  |x|^{-N} \quad \text{as} \quad |x|\rightarrow \infty.
\end{align*}
\end{enumerate}
\end{proposition}

For the kernel $K_{b}$ we utilize certain estimates for the Littlewood-Paley projections of the kernel.  To this end, we introduce these projections.  Let $\Phi:\mathbb{R}^n \to \mathbb{R}$ be a smooth, compactly supported function such that 
\begin{align*}
\operatorname*{supp}\Phi \subset B(0,2),\\
\Phi \equiv 1 \text{ on }
B(0,1),
\end{align*}
and define $\phi(\xi)=\Phi(\xi)-\Phi(2\xi)$.  Further define $\phi_j(\xi)= \phi(2^{-j}\xi)$.  Then one has
\begin{align}\label{partition of unity}
    1=\sum_{j \in \mathbb{Z}} \phi_j(\xi), \quad \xi \in \mathbb{R}^n \setminus \{0\}.
\end{align}
Define
\begin{align}
m_b(\xi) = \psi(\xi) (2\pi |\xi|)^{-b}e^{i|\xi|},
\end{align}
the multiplier operator arising in the definition \eqref{main-operator} and for $j\in \mathbb{Z}$ let $m_{b}^j(\xi)= m_{b}(\xi) \phi(2^{-j}\xi)$. For each $j\in \mathbb{Z}$, we write 
\begin{align}\label{def-Tj}
 \widetilde{T}_{b}^jf(x)=K_{b}^{j}\ast f(x)   
\end{align}
where $\widehat{K}_{b}^j(\xi)=m_{b}^j(\xi)$ for each $j$.

Then we have the following result.
\begin{lemma}\label{kerne--esti-1}
  Let $n\geq 2$ and $q\in [1,\infty]$ and $b\in \mathbb{R}$. Then there exists a constant $C=C(n,p,b)$ such that
  \begin{align*}
\|\partial^{\alpha}K_{b}^{j}\|_{L^q(\mathbb{R}^n)}\leq C 2^{j(\frac{n+1}{2}-\frac{1}{q}-b)} 2^{j|\alpha|}
  \end{align*}
  for all $\alpha\in \mathbb{N}^n$ and $j\in \mathbb{Z}$.
\end{lemma}
\begin{proof}
    The proof of the above lemma for $|\alpha|=0$ follows from Lemma 3.4 of \cite{kMST-I}. For $|\alpha|\neq 0$, note that
    \begin{align*}
        \partial^{\alpha}_x K_b^j(x)=\int_{\mathbb{R}^n} (2\pi i)^{|\alpha|} \xi^{\alpha} m_{b}^j(\xi) e^{2\pi i \xi\cdot x} \, d\xi,
    \end{align*}
    Therefore, the proof for $|\alpha| \neq 0$ follows similarly, and as a result, we obtain a growth of $2^{j|\alpha|}$ in the estimate.

    This completes the proof of the lemma.
\end{proof}

In our proofs we make use of another useful property of the Littlewood-Paley projectors, that
\begin{align*}
\phi_{j-1}(\xi)+\phi_j(\xi)+\phi_{j+1}(\xi)=1 \text{ for } \xi \in \operatorname*{supp}\phi_j.
\end{align*}
This identity allows one to introduce a second projector
\begin{align}\label{second_projector}
 \widetilde{T}_{b}^jf(x)=K_{b}^{j}\ast f(x) = K_{b}^{j}\ast \widetilde{P}_jf(x)  
\end{align}
for 
\begin{align}
 \widehat{\widetilde{P}_jf}(\xi) = \left(\phi_{j-1}(\xi)+\phi_j(\xi)+\phi_{j+1}(\xi)\right)\widehat{f}(\xi).
\end{align}
This fact will be useful in the sequel.

We next recall the definition of the atomic spaces with dimensional stability $\beta\in (0,n]$.  To this end we first introduce the notion of $\beta$-atom:
\begin{definition}[$\beta$-atom]
For $\beta \in (0,n]$,  we say that $a \in M_b(\mathbb{R}^n)$ is a $\beta$-atom if there exists a cube $Q\subset \mathbb{R}^d$ with sides parallel to the coordinate axes such that
\begin{enumerate}
\item \label{supp}$\operatorname*{supp} a \subset Q$;
\item \label{ave} $a(Q) = 0$;
\item $\esssup_{x \in \mathbb{R}^n, t>0} |t^{(n-\beta)/2} e^{t\Delta} a (x)|  \leq \frac{1}{l(Q)^\beta}\label{linfty} $;
\item \label{l1} $|a|(\mathbb{R}^n) \leq 1$.
\end{enumerate}
\end{definition}
\noindent

Here and in the sequel, we use $e^{t\Delta}f$ to denote the heat extension of a distribution $f$.  For $f=\mu$ a locally finite Radon measure, one has
\begin{align}\label{heat_extension}
e^{t\Delta}\mu(x) = (4\pi t)^{-\frac{n}{2}}\int\limits_{\mathbb{R}^n}e^{-\frac{|x-y|^2}{4t}}\,d\mu(y)=p_t\ast\mu(x),\qquad x\in \R^n.
\end{align} 

The dimension stable spaces are defined as follows.
\begin{definition}[$\beta$-atomic space]
 Let $\beta \in [0,n]$.  Define the atomic space of dimension $\beta$ by
\begin{align*}
DS_\beta&(\mathbb{R}^n):=\\ 
&\left\{\mu \in M_b(\mathbb{R}^n)\colon \mu = \lim_{l\to \infty}\sum_{i=1}^{l} \lambda_{i,l} a_{i,l},    \text{ $a_{i,l}$ $\beta$-atoms, } \limsup_{l \to \infty} \sum_{i=1}^l |\lambda_{i,l}|<+\infty \right\}.
\end{align*}
Here the convergence is intended weakly-star as measures,
\begin{align*}
\int \varphi \;d\mu= \lim_{l \to \infty}  \sum_{i=1}^{l} \lambda_{i,l} \int \varphi \;da_{i,l},
\end{align*}
for all $\varphi \in C_0(\mathbb{R}^n)$, and the space is equipped with the norm
\begin{align*}
\|\mu \|_{DS_\beta(\mathbb{R}^n)}:= \inf \left\{\limsup_{l \to \infty} \sum_{i=1}^l |\lambda_{i,l}| \colon \mu = \lim_{l\to \infty}\sum_{i=1}^{l} \lambda_{i,l} a_{i,l} \right\}.
\end{align*}
\end{definition}

\begin{proposition}\label{smooth_atoms}
Let $\mu \in DS_\beta(\mathbb{R}^n)$.  There exists a sequence of atoms $\{\widetilde{a}_{i,l}\} \subset C^\infty_c(\mathbb{R}^n)$ and scalars $\{\widetilde{\lambda}_{i,l}\}$ such that
\begin{align*}
\mu = \lim_{l \to \infty} \sum_{i=1}^l \widetilde{\lambda}_{i,l} \widetilde{a}_{i,l}
 \end{align*}
such that 
\begin{align*}
\limsup_{l \to \infty} \sum_{i=1}^l |\widetilde{\lambda}_{i,l}|  \lesssim \|\mu\|_{DS_\beta(\mathbb{R}^n)}.
\end{align*}
\end{proposition}

\begin{proof}
Let $\rho$ be a smooth function which is supported in $B(0,1)$ and satisfies $\int \rho=1$. For $\epsilon>0$, let $\rho_{\epsilon}(x)= \epsilon^{-n} \rho(x/\epsilon)$. If $a$ is a $\beta$-atom with $\supp (a) \subset Q$, then for $\epsilon < \ell(Q)$ standard manipulation of convolution gives
  \begin{align*}
  \supp(a\ast \rho_{\epsilon}) &\subset 2Q;\\
  \int_{2Q} a\ast \rho_{\epsilon}\;dx &=0;\\
       t^{(n-\beta)/2}\|e^{t\Delta} (a\ast \rho_{\epsilon})\|_{L^{\infty}(\mathbb{R}^n)} &\leq t^{(n-\beta)/2} \|e^{t\Delta} a\|_{L^{\infty}(\mathbb{R}^n)} \leq \frac{2^\beta}{l(2Q)^\beta};\\
        \|a\ast \rho_{\epsilon}\|_{L^{1}(\mathbb{R}^n)} &\leq |a|(\mathbb{R}^n).
  \end{align*}
 In particular, $(a\ast \rho_{\epsilon})/2^\beta \in C^\infty_c(\mathbb{R}^n)$ is a $\beta$-atom and $\lim_{\epsilon \to 0^+} a\ast \rho_{\epsilon}=a$.  Thus, if
  \begin{align*}
\mu = \lim_{l \to \infty} \sum_{i=1}^l \lambda_{i,l} a_{i,l}
 \end{align*}
 for some sequence of $\beta$-atoms $\{a_{i,l}\}$, then
   \begin{align*}
\mu = \lim_{l \to \infty} \lim_{\epsilon \to 0^+} \sum_{i=1}^l \lambda_{i,l} a_{i,l}\ast \rho_\epsilon = \lim_{l \to \infty} \sum_{i=1}^l \lambda_{i,l} a_{i,l,\epsilon_{i,l}}
 \end{align*}
for a suitable selection of $\epsilon_{i,l}>0$ by diagonalization.  Here we use the notation $a_{i,l,\epsilon_{i,l}} = (a_{i,l} \ast\rho_{\epsilon_{i,l}}) \in C^\infty_c(\mathbb{R}^n)$.  The claimed approximation then follows because
 \begin{align*}
\mu = \lim_{l \to \infty} \sum_{i=1}^l \lambda_{i,l}2^\beta a_{i,l,\epsilon_{i,l}}/2^\beta =\lim_{l \to \infty} \sum_{i=1}^l \widetilde{\lambda}_{i,l}\widetilde{a}_{i,l}
 \end{align*}
with $\widetilde{a}_{i,l} = a_{i,l,\epsilon_{i,l}} /2^\beta \in C^\infty_c(\mathbb{R}^n)$ an atom and $\widetilde{\lambda}_{i,l} = \lambda_{i,l}2^\beta$ for each $l \in \mathbb{N}$ and $i=1,\ldots,l$.  This completes the proof of the proposition.
\end{proof}

In the next two lemmas, we estimate the $L^2$ norm of the Riesz potential acting on these $\beta$-atoms, which will be useful for controlling the small atoms in the proofs of the main theorems.
\begin{lemma}\label{Riesz-potential-esti}
Let  $n\geq 3$ and $p\in [1,2]$.  Let $a$ be a $\beta$-atom adapted to a cube $Q$ with side length $\ell(Q)$. Then for $\beta > \frac{n-1}{2}+\frac{1}{p}$, there exists a constant $C=C(n,p, \beta)$ such that 
\begin{align*}
\|I_{b_p}(a)\|_{L^2(\mathbb{R}^n)}\leq C \ell(Q)^{\frac{1}{2}-\frac{1}{p}}.
\end{align*}
\end{lemma}
\begin{proof}

By the semigroup property and the heat kernel representation of $I_{b_p}$, we write
\begin{align}\label{Riesz-potential}
\nonumber \|I_{b_p}(a)\|_{L^2(\mathbb{R}^n)}^2 & = \int_{\mathbb{R}^n} I_{b_p}(a)(x) \overline{I_{b_p}(a)(x)} \, dx\\
\nonumber & =  \int_{\mathbb{R}^n} I_{2b_p}(a)(x) \overline{a(x)} \, dx\\
\nonumber & =  \int_{\mathbb{R}^n} \left(\int_{0}^{\infty} t^{b_p-1} p_t\ast a(x) \, dt \right) \overline{a(x)}\, dx\\
\nonumber & =  \int_{0}^{\infty} t^{b_p-1} \left( \int_{\mathbb{R}^n}  p_t\ast a(x) \overline{a(x)} \, dx  \right) \, dt\\
\nonumber & = \int_{0}^{\infty} t^{b_p-1} \int_{\mathbb{R}^n} \left|p_{\frac{t}{2}}\ast a(x) \right|^2 \, dx \, dt\\
& =  \int_{0}^{\infty} t^{b_p-1} \int_{\mathbb{R}^n} p_{\frac{t}{2}}\ast a(x)~   \overline{p_{\frac{t}{2}}\ast a(x)} \,dx\, dt\\
\nonumber & \lesssim  \int_{0}^{\infty} t^{b_p-1} \|p_t\ast a\|_{L^1(\mathbb{R}^n)}\|p_t\ast a\|_{L^\infty(\mathbb{R}^n)}  \, dt \\
\nonumber &\leq \int_{0}^{\ell(Q)^2} t^{b_p-1} \|p_t\ast a\|_{L^1(\mathbb{R}^n)}\|p_t\ast a\|_{L^\infty(\mathbb{R}^n)}  \, dt \\
\nonumber \quad&+ \int_{\ell(Q)^2}^\infty t^{b_p-1} \|p_t\ast a\|_{L^1(\mathbb{R}^n)}\|p_t\ast a\|_{L^\infty(\mathbb{R}^n)}  \, dt\\
\nonumber &=:I+II.
\end{align}

To estimate $I$, we recall estimates (3.4) and (3.5) of \cite{Spector-Stolyarov}: 
\begin{align*}
 \|p_t\ast a\|_{L^\infty(\mathbb{R}^n)}  &\lesssim  \frac{1}{t^{\frac{n-\beta}{2}}\ell(Q)^{\beta}}\\
  \|p_t\ast a\|_{L^1(\mathbb{R}^n)}  &\lesssim \max \left\{ \frac{\ell(Q)^{n-\beta}}{t^{\frac{n-\beta}{2}}}, \frac{\ell(Q)}{t^{\frac{1}{2}}}\right\}.   
\end{align*}

For $n\geq 3$, one has $\frac{n-1}{2}+\frac{1}{p} < n-1$.  We may therefore assume $\beta \leq n-1$, as the inequality for this range of $\beta$ implies the inequality for all larger $\beta$ by the nested properties of $DS_\beta(\mathbb{R}^n)$.  In particular, we have   
\begin{align}\label{small-t}
  I\nonumber &= \int_{0}^{\ell(Q)^2} t^{b_p-1} \|p_t\ast a\|_{L^1(\mathbb{R}^n)}\|p_t\ast a\|_{L^\infty(\mathbb{R}^n)}   \, dt \\
 \nonumber &\lesssim  \ell(Q)^{n-2\beta} \int_{0}^{\ell(Q)^2} t^{b_p-1-n-\beta}   \, dt \\
 &=  \ell(Q)^{n-2\beta}  \int_{0}^{\ell(Q)^2} t^{b_p-1-(n-\beta)}   \, dt\\ 
 \nonumber &\lesssim  \ell(Q)^{n-2\beta} \ell(Q)^{2b_p-2n+2\beta} \\
\nonumber &\lesssim \ell(Q)^{2b_p-n} = \ell(Q)^{1-\frac{2}{p}}. 
\end{align}
This requires $b_p-1-(n-\beta)>-1$, which holds because of the assumption
\begin{align*}
\beta> n-b_p=\frac{n-1}{2}+\frac{1}{p}.
\end{align*}

For $II$ we use the $L^\infty$ bound for convolution and the improved decay in $L^1$ bound recalled above:
\begin{align}
II\nonumber &\lesssim \int_{\ell(Q)^2}^{\infty} t^{b_p-1-\frac{n}{2}}  \|p_t\ast a\|_{L^1(\mathbb{R}^n)} \, dt \\   
\nonumber &\lesssim  \int_{\ell(Q)^2}^{\infty}  t^{b_p-1-\frac{n}{2}}  \frac{l(Q)}{t^{\frac{1}{2}}} \, dt\\
\nonumber&\lesssim  \ell(Q)^{2b_p-n},
\end{align}
where the finiteness of the integral is justified by the fact that $2b_p-n-1<0$.

Finally, the combination of the estimates for $I$ and $II$ and \eqref{Riesz-potential} yield
\begin{align*}
    \|I_{b_p}(a)\|_{L^2(\mathbb{R}^n)} \lesssim \ell(Q)^{\frac{1}{2}-\frac{1}{p}}.
\end{align*}

This completes the proof of the lemma.
\end{proof}

\begin{lemma}\label{Riesz-potential-esti-n=2}
Let  $n=2$ and $p\in (1,2)$.  Let $a$ be a $\beta$-atom adapted to a cube $Q$ with side length $\ell(Q)$. Then for $\beta > \frac{2}{p}$, there exists a constant $C=C(p, \beta)$ such that 
\begin{align*}
\|I_{b_p}(a)\|_{L^2(\mathbb{R}^2)}\leq C \ell(Q)^{\frac{1}{2}-\frac{1}{p}}.
\end{align*}
\end{lemma}

\begin{proof}
We argue as in the preceding proof, mutatis mutandis.  For $I$, the $L^1$ bounds on the convolution are inverted and we obtain
\begin{align*}
  I &= \int_{0}^{\ell(Q)^2} t^{b_p-1}  \|p_t\ast a\|_{L^1(\mathbb{R}^2)}\|p_t\ast a\|_{L^\infty(\mathbb{R}^2)}   \, dt \\
 &\lesssim  \int_{0}^{\ell(Q)^2} t^{b_p-1} \frac{1}{t^{\frac{2-\beta}{2}}l(Q)^{\beta}} \frac{l(Q)}{t^{\frac{1}{2}}}  \, dt \\
&\lesssim  \ell(Q)^{2b_p-2},
\end{align*}
which requires
\begin{align*}
\beta>\frac{2}{p}.
\end{align*}
The alternate to $II$ is then 
\begin{align}
II\nonumber &\lesssim \int_{\ell(Q)^2}^{\infty} t^{b_p-1}  \|p_t\ast a\|_{L^1(\mathbb{R}^2)}  \|p_t\ast a\|_{L^\infty(\mathbb{R}^2)}  \, dt \\ 
\nonumber &\lesssim  \int_{\ell(Q)^2}^{\infty}  t^{b_p-1} \frac{l(Q)^{2-\beta}}{t^{\frac{2-\beta}{2}}} \frac{1}{t} \, dt\\
\nonumber& \lesssim  \ell(Q)^{2b_p-2},
\end{align}
which holds because $\beta\leq 2 < \frac{2}{p}+1$.  This completes the proof of the lemma.
\end{proof}

\section{The estimate for \texorpdfstring{$2\leq p <\infty$}{p bigger 2}}\label{p_larger_than_2}
In this section, we will provide the proof of Theorem \ref{Thm-p-bigger-2}. To this end, it will be useful to first prove several preliminary results.

\begin{proposition}\label{prop-L2-gamma-1}
 Let $\beta\in (0, n]$. Then there exists a constant $C=C(n,\beta)$ such that 
 \begin{align*}
     \|T_{\frac{n}{2}}\mu\|_{L^2(\mathbb{R}^n)} \leq C \|\mu\|_{DS_{\beta}(\mathbb{R}^n)}
 \end{align*}
for all $\mu \in DS_{\beta}(\mathbb{R}^n)$.
\end{proposition}
 \begin{proof}
     Let $\mu \in DS_{\beta}(\mathbb{R}^n)$. Theorem A of \cite{Spector-Stolyarov} implies the bound
     \begin{align*}
         \|I_{\frac{n}{2}}\mu\|_{L^2(\mathbb{R}^n)} \leq C \|\mu\|_{DS_{\beta}(\mathbb{R}^n)}
     \end{align*}
     for some constant $C>0$. This inequality, in combination with Plancherel's theorem yields
     \begin{align*}
         \|T_{ \frac{n}{2}}\mu\|_{L^2(\mathbb{R}^n)}^2= & \int_{\mathbb{R}^n} (1+4\pi^2|\xi|^2)^{-n/2} |\widehat{\mu}(\xi)|^2 \, d\xi \\
         \leq & \int_{\mathbb{R}^n} \left|I_{\frac{n}{2}}\mu(x) \right|^2 \, dx\\
         \leq & C \|\mu\|_{DS_{\beta}(\mathbb{R}^n)}^2,
     \end{align*}
   which completes the proof of the theorem.
 \end{proof}

\begin{proposition}\label{prop-BMO-gamma-1}
    Let $\beta\in (0, n]$. Then there exists a constant $C>0$ such that 
    \begin{align*}
        \|T_{\frac{n+1}{2}}\mu\|_{BMO(\mathbb{R}^n)} \leq C \|\mu\|_{DS_{\beta}(\mathbb{R}^n)}
    \end{align*}
    for all $\mu \in DS_{\beta}(\mathbb{R}^n)$.
\end{proposition}
\begin{proof}
First, note that using \eqref{Decomposition-operator} we can write
\begin{align*}
T_{\frac{n+1}{2}}\mu &= \overline{T}_{\frac{n+1}{2}}\mu + \nu_{\frac{n+1}{2}} \ast \widetilde{T}_{\frac{n+1}{2}}\mu,\\
&=\overline{T}_{\frac{n+1}{2}}\mu +  \widetilde{T}_{\frac{n+1}{2}}(\nu_{\frac{n+1}{2}} \ast\mu).
\end{align*}
where $\overline{T}_{\frac{n+1}{2}}\mu$ is an Fourier multiplier operator defined in \eqref{goodkernel} which maps $L^1(\mathbb{R}^n)$ to  $L^p(\mathbb{R}^n)$ for all $1\leq p \leq \infty$ and $\nu_{\frac{n+1}{2}}$ is bounded measure on $\mathbb{R}^n$. 
As $DS_{\beta}(\mathbb{R}^n)\subset L^1(\mathbb{R}^n)$ and $L^{\infty}(\mathbb{R}^n)\subset BMO(\mathbb{R}^n)$, it is enough to prove boundedness result for the operator $\widetilde{T}_{\frac{n+1}{2}}$. The required boundedness of $\widetilde{T}_{\frac{n+1}{2}}$ follows from \cite{Miyachi-singular}*{Theorem 4.2}, which says that $\widetilde{T}_{\frac{n+1}{2}}$ maps from $L^1(\mathbb{R}^n)$ to $BMO(\mathbb{R}^n)$.     
This completes the proof of the proposition.
\end{proof}

The proof of Theorem \ref{Thm-p-bigger-2} is by complex interpolation.  That one can implement complex interpolation for an analytic family of operators with $BMO$ as an endpoint is trivial to the masters \cite{FS, Stein}, though for convenience of the reader we here provide several details not presented in those works (or other places in the literature as far as the authors are aware). The first of these additional details is the following lemma, some manifestation of which seems to be used implicitly in the proofs in \cite{FS, Stein}, and whose idea of proof was communicated to us by Po Lam Yung.

\begin{lemma}\label{pointwise_approximation}
Let $g$ be a measurable function such that the Fefferman-Stein sharp maximal function satisfies
\begin{align*}
M^\#g(x):= \sup_{r>0}\fint_{B(x,r)} \left|g(y)-\left(g\right)_{B(x,r)} \right|\, dy<+\infty
\end{align*}
for almost every $x \in \mathbb{R}^n$.  For each $l \in \mathbb{N}$, there exists a measurable function $r_\ell:\mathbb{R}^n \to [1/\ell,\ell]\cap \mathbb{Q}$ and a measurable function $\eta_\ell :\mathbb{R}^n\times \mathbb{R}^n \to \mathbb{C}$ such that for any  $f \in L^1_{loc}(\mathbb{R}^n)$ the function
\begin{align}\label{measurable_approximation}
U_\ell(f)(x):=\fint_{B(x,r_\ell(x))} \left[f(y)-\left(f\right)_{B(x,r_\ell(x))} \right] \eta_\ell(x,y)\, dy
\end{align}
is a measurable function and
\begin{align}\label{convergence}
M^\#g(x) = \lim_{l \to \infty} U_\ell(g)(x) 
\end{align}
for pointwise almost every $x \in \mathbb{R}^n$.
\end{lemma}

\begin{proof}
For each $\ell \in \mathbb{N}$, define $\Lambda_\ell:= [1/\ell,\ell]\cap \mathbb{Q}$ and the corresponding truncated maximal function
\begin{align*}
M^\#_\ell g(x) = \sup_{r \in \Lambda_\ell}  \fint_{B(x,r)} |g(y)-\left(g\right)_{B(x,r)}|\, dy.
\end{align*}

Let $\{r_i^\ell\}$ be an enumeration of $\Lambda_\ell$ and define $r_\ell(x) = r_1^\ell$ if 
\begin{align}\label{almost_optimizer}
\fint_{B(x,r_1^\ell)} |g(y)-\left(g\right)_{B(x,r_1^\ell)}|\, dy > \left(1-\frac{1}{\ell}\right) M^\#_\ell g(x).
\end{align}
Note that the set of all $x$ for which $r_\ell(x) = r_1^\ell$ is measurable, by the measurability of the two maps 
\begin{align*}
x &\mapsto \fint_{B(x,r_1^\ell)} |g(y)-\left(g\right)_{B(x,r_1^\ell)}|\, dy\\
x &\mapsto \sup_{r \in \Lambda_\ell} \fint_{B(x,r)} |g(y)-\left(g\right)_{B(x,r)}|\, dy.
\end{align*}
Denote by $\Omega_1^\ell$ the set of all such $x$ selected in this step.  Next consider the set of all $x \notin \Omega_1^\ell$ for which
\begin{align*}
\fint_{B(x,r_2^\ell)} |g(y)-\left(g\right)_{B(x,r_2^\ell)}|\, dy > \left(1-\frac{1}{\ell}\right) M^\#_\ell g(x).
\end{align*}
This is again a measurable set. Inductively, one defines $r_k^\ell, \Omega_k^\ell$  in terms of $k$, the smallest positive integer such that
\begin{align}\label{almost_optimizer_k}
\fint_{B(x,r_k^\ell)} |g(y)-\left(g\right)_{B(x,r_k^\ell)}|\, dy > \left(1-\frac{1}{\ell}\right) M^\#_\ell g(x).
\end{align}
As $r_\ell(x)$ takes only countably many values, i.e. 
\begin{align*}
r_\ell(x) = \sum_{k=1}^{\infty} r^\ell_k 1_{\Omega^\ell_k}(x),
\end{align*}
we see that $r_\ell(x)$ is measurable as a function of $x$. Finally define 
\begin{align*}
\eta_\ell(x,y):= \frac{g(y)-\left(g\right)_{B(x,r_\ell)}}{|g(y)-\left(g\right)_{B(x,r_\ell)}|}
\end{align*}
if $|g(y)-\left(g\right)_{B(x,r_\ell)}| \neq 0$ and $0$ otherwise.  The measurability of $\eta_\ell$ follows from the measurability of $(x,y) \mapsto g(y)-\left(g\right)_{B(x,r_\ell)}$.  In particular, for any locally integrable function $f$,  the map defined by \eqref{measurable_approximation} is a measurable function.  Therefore it finally remains to show that we have the claimed convergence \eqref{convergence}.  Noting that
\begin{align*}
M^\# g(x) = \lim_{\ell \to \infty} M^\#_\ell g(x),
\end{align*}
we have
\begin{align*}
M^\# g(x) &= \lim_{\ell \to \infty} M^\#_\ell g(x) \\
&\leq \liminf_{\ell \to \infty} \frac{\ell}{\ell-1} \fint_{B(x,r_\ell(x))} \left|g(y)-\left(g\right)_{B(x,r_\ell(x))} \right| \, dy\\
&=\liminf_{\ell \to \infty}  \fint_{B(x,r_\ell(x))} \left[g(y)-\left(g\right)_{B(x,r(x))} \right] \eta_\ell(x,y)\, dy\\
&=\liminf_{\ell \to \infty}  U_\ell(g)(x)\\
&=\limsup_{\ell \to \infty} U_\ell(g)(x)\\
&\leq M^\# g(x).
\end{align*}

\end{proof}

\begin{proof}[Proof of Theorem \ref{Thm-p-bigger-2}]
Let $\beta\in (0,n]$. We know by Theorem \ref{M_L1} that the operator $T_b$ is bounded from $L^1(\mathbb{R}^n)$ to  $L^p(\mathbb{R}^n)$ for $1\leq p \leq \infty$ and  for $b>b_p=\frac{n+1}{2}-\frac{1}{p}$. Since $DS_{\beta}(\mathbb{R}^n)\subset L^1(\mathbb{R}^n)$, it is enough to prove the Theorem \ref{Thm-p-bigger-2} for $b=b_p$. For brevity of notation we suppress the dependence in $p$ and write $b$ for ${b_p}$.
By Proposition \ref{smooth_atoms}, it suffices to prove that there exists a constant $C>0$ independent of the atom such that
    \begin{align*}
        \|T_b(a)\|_{L^p(\mathbb{R}^n)} \leq C
    \end{align*}
for any $\beta$-atom $a \in C^\infty_c(\mathbb{R}^n)$.  

We argue by complex interpolation with an adaptation of the proof in \cite[5.2 on p.~175]{Stein}.  Let $S$ denote the closed strip $\{z\in \mathbb{C}: 0\leq Re(z)\leq 1\}$.  For $f \in L^2(\mathbb{R}^n)$, define the family of operators $\{\mathcal{T}_z\}_{z\in S}$ by 
    \begin{align*}
        \mathcal{T}_zf(x)= \int_{\mathbb{R}^n}  \widehat{f}(\xi) (1+|\xi|^2)^{-\frac{n+z}{4}} e^{i|\xi|} e^{2\pi i \xi\cdot x} \, d\xi.
    \end{align*}
    Then
    \begin{align*}
   \mathcal{T}_z :  L^2(\mathbb{R}^n) \to  L^2(\mathbb{R}^n),
    \end{align*}
and for $f, h \in L^2(\mathbb{R}^n)$ the map
\begin{align}\label{analytic_one}
    z \mapsto \int_{\mathbb{R}^n}  \mathcal{T}_z(f) h\;dx
    \end{align}
is admissible in the sense of \cite[V.4 on p.~205] {SteinWeiss}, i.e.~it is continuous on $S$, analytic in the interior, and setting $z=s+it$, the quantity
\begin{align}\label{analytic_two}
e^{-|t|} \log \left|\int_{\mathbb{R}^n}  \mathcal{T}_z(f) h\;dx\right|
    \end{align}
is uniformly bounded above.

For $\ell \in \mathbb{N}$, let $r_\ell$ be the maps given by \eqref{pointwise_approximation} applied to the function $g=\mathcal{T}_\theta f= T_{\frac{n+\theta}{2}}f$ and consider the associated maps
\begin{align}\label{sharp_linearization}
    U_\ell(\mathcal{T}_zf)(x)= \fint_{B(x,r_\ell(x))} \left[\mathcal{T}_zf(y)-\left(\mathcal{T}_zf\right)_{B(x,r_\ell(x))} \right] \eta_\ell(x,y)\, dy.
\end{align}
Observe that the definition of the centered Fefferman-Stein sharp maximal function along with Lemma \ref{pointwise_approximation} ensure that
\begin{align}
    \left| U_\ell (\mathcal{T}_zf)(x)\right|&\leq  M^{\#}(\mathcal{T}_zf)(x), \quad \text{for all } x \in \mathbb{R}^n, z \in [0,1],\label{sharp-maximal-domination}\\
    \lim_{\ell \to \infty} \left| U_\ell (\mathcal{T}_\theta f)(x)\right| &= M^{\#}(\mathcal{T}_\theta f)(x), \quad \text{for a.e. } x \in \mathbb{R}^n.\label{sharp-maximal-attained}
\end{align}

We next show that for each $\ell \in \mathbb{N}$ one has
\begin{align}\label{intermediate_inequality}
    \|U_\ell(\mathcal{T}_\theta f)\|_{L^p(\mathbb{R}^n)} \leq C\|T_{ n/2}f\|_{L^2(\mathbb{R}^n)}^{1-\theta} \|f\|_{L^1(\mathbb{R}^n)}^\theta, \quad \theta= 1-\frac{2}{p}
\end{align}
for any simple function $f$. It is sufficient to prove 
\begin{align*}
    \left| \int_{\mathbb{R}^n} U_\ell(\mathcal{T}_\theta f)(x) \overline{h(x)}\;dx \right| < C\|T_{ n/2}f\|_{L^2(\mathbb{R}^n)}^{1-\theta} \|f\|_{L^1(\mathbb{R}^n)}^\theta
\end{align*}
where $h$ is an arbitrary simple function with $\|h\|_{L^{p'}(\mathbb{R}^n)}\leq 1$ and $C$ is a positive constant independent of $f$.  Let $h= \sum_{j}a_j e^{i \alpha_j}\chi_{E_j}$ for $a_j \in \mathbb{R}^+, \alpha_j \in \mathbb{R}$, and $z\in S$, set
\begin{align*}
    h_z=\sum_{j} (a_j)^{\frac{(1+z)p'}{2}} e^{i \alpha_j} \chi_{E_j}. 
\end{align*}

Then one can easily verify $h_\theta=h$ and 
\begin{align*}
    \|h_{it}\|_{L^2(\mathbb{R}^n)} &= \sum_{j} (a_j)^{p'} |E_j|= \|h\|_{L^{p'}(\mathbb{R}^n)}\leq 1,\\
    \|h_{1+it}\|_{L^1(\mathbb{R}^n)} &= \|h\|_{L^{p'}(\mathbb{R}^n)}\leq 1.
\end{align*}

Define $I_\ell(z): S \mapsto \mathbb{C}$ by
\begin{align*}
    I_\ell(z)= \int_{\mathbb{R}^n} U_\ell(\mathcal{T}_z f)(x) \overline{h_z(x)}\;dx.
\end{align*}
As in \cite[p.~176]{Stein}, for each $\ell \in \mathbb{N}$, $I_\ell(z)$ is holomorphic.  To see this, one begins by expanding
\begin{align*}
    I_\ell(z)&= \sum_j (a_j)^{\frac{(1+z)p'}{2}} e^{-i \alpha_j} \int_{\mathbb{R}^n} \fint_{B(x,r_\ell(x))} \left[\mathcal{T}_zf(y)-\left(\mathcal{T}_zf\right)_{B(x,r_\ell(x))} \right] \eta_\ell(x,y)\, dy \;\chi_{E_j}(x)\,dx.
\end{align*}
For each $j$, we note that
\begin{align*}
& \int_{\mathbb{R}^n} \fint_{B(x,r_\ell(x))} \left[\mathcal{T}_zf(y)-\left(\mathcal{T}_zf\right)_{B(x,r_\ell(x))} \right] \eta_\ell(x,y)\, dy \;\chi_{E_j}(x)\,dx\\
    &=\int_{\mathbb{R}^n} \int_{\mathbb{R}^n} \frac{1}{c_n r_\ell(x)^n} \chi_{B(x,r_\ell(x))}(y) \mathcal{T}_zf(y) \eta_\ell(x,y) \chi_{E_j}(x)\;dydx\\
    & - \int_{\mathbb{R}^n} \left( \int_{\mathbb{R}^n} \frac{\chi_{B(x,r_\ell(x))}(y)}{c_n r_\ell(x)^n} \mathcal{T}_zf(y) \, dy\right) \int_{\mathbb{R}^n} \frac{\chi_{B(x,r_\ell(x))}(w)}{c_n r_\ell(x)^n} \eta_\ell(x,w) \, dw \, \chi_{E_j}(x) \, dx\\
    & = \int_{\mathbb{R}^n} \mathcal{T}_zf(y) \left[ \int_{\mathbb{R}^n} \chi_{E_j}(x) \left(\frac{\chi_{B(x,r_\ell(x))}(y)}{c_n r_\ell(x)^n}  \eta_\ell(x,y) -\right.\right.\\
    & \quad \quad \quad \quad \quad \quad \quad \quad \left.\left.\frac{\chi_{B(x,r_\ell(x))}(y)}{c_n r_\ell(x)^n} \int_{\mathbb{R}^n} \frac{\chi_{B(x,r_\ell(x))}(w)}{c_n r_\ell(x)^n} \eta_\ell(x,w) \, dw \right)\, dx \right] \, dy\\
    & = \int_{\mathbb{R}^n} \mathcal{T}_zf(y) \, G_j(y) \, dy
\end{align*}
where $G_j(y)$ is given by
\begin{align*}
\int_{\mathbb{R}^n} \chi_{E_j}(x) \left(\frac{\chi_{B(x,r_\ell(x))}(y)}{c_n r_\ell(x)^n}  \eta_\ell(x,y) -\frac{\chi_{B(x,r_\ell(x))}(y)}{c_n r_\ell(x)^n} \int_{\mathbb{R}^n} \frac{\chi_{B(x,r_\ell(x))}(w)}{c_n r_\ell(x)^n} \eta_\ell(x,w) \, dw \right)\, dx.
\end{align*}
Note that $G_j$ is in $L^{\infty}_0(\mathbb{R}^n)$. The admissibility of \eqref{analytic_one} thus gives the admissibility of
\begin{align*}
z\mapsto \int_{\mathbb{R}^n} \mathcal{T}_zf(y) \, G_j(y) \, dy,
\end{align*}
and therefore also $I_\ell(z)$ as a finite sum of products of these functions with bounded holomorphic functions. 

By H\"older's inequality we have the bound
\begin{align*}
    \left|I_\ell(it)\right| &\leq  \|U_\ell(\mathcal{T}_{it}f)\|_{L^2(\mathbb{R}^n)} \|h_{it}\|_{L^2(\mathbb{R}^n)}\\
    &\lesssim  \|M^{\#}(\mathcal{T}_{it}f)\|_{L^2(\mathbb{R}^n)}\\
    &\lesssim  \|\mathcal{T}_{it}f\|_{L^2(\mathbb{R}^n)}\\
    &\leq\|(I-\Delta)^{\frac{it}{2}}\|_{L^2(\mathbb{R}^n)\rightarrow L^2(\mathbb{R}^n) } \|T_{ n/2}f\|_{L^2(\mathbb{R}^n)}\\
    &\lesssim \|T_{ n/2}f\|_{L^2(\mathbb{R}^n)}.
\end{align*}

Similarly, using Holder's inequality and Proposition \ref{prop-BMO-gamma-1}, we obtain
\begin{align*}
    \left|I_\ell(1+it)\right|&\leq  \|U_\ell(\mathcal{T}_{1+it}f)\|_{L^\infty(\mathbb{R}^n)} \|h_{1+it}\|_{L^1(\mathbb{R}^n)}\\
    &\leq  \|M^{\#}(\mathcal{T}_{1+it}f)\|_{L^\infty(\mathbb{R}^n)}\\
    &\leq  \|\mathcal{T}_{1+it}a\|_{BMO(\mathbb{R}^n)}\\
    &\leq  \|(I-\Delta)^{\frac{it}{2}}\|_{BMO(\mathbb{R}^n)\rightarrow BMO(\mathbb{R}^n) } \|T_{ \frac{n+1}{2}}f\|_{BMO(\mathbb{R}^n)}\\
    &\leq  \|(I-\Delta)^{\frac{it}{2}}\|_{BMO(\mathbb{R}^n)\rightarrow BMO(\mathbb{R}^n) }\|f\|_{L^1(\mathbb{R}^n)}.
\end{align*}
Note that the Mihlin-H\"ormander multiplier theorem implies
\begin{align}
\|(I-\Delta)^{\frac{it}{2}}\|_{BMO(\mathbb{R}^n)\rightarrow BMO(\mathbb{R}^n) }=\|(I-\Delta)^{\frac{it}{2}}\|_{\mathcal{H}^1(\mathbb{R}^n)\rightarrow \mathcal{H}^1(\mathbb{R}^n) } \lesssim (1+|t|)^{n/2+1}
\end{align}
and therefore, by \cite[Lemma 4.2 on p.~206]{SteinWeiss} we deduce that
\begin{align*}
\log |I_\ell(\theta)| &\leq \log C\|T_{ n/2}f\|_{L^2(\mathbb{R}^n)}^{1-\theta}+  \log \|f\|_{L^1(\mathbb{R}^n)}^{\theta} \\
&\quad + \frac{1}{2}\sin(\pi \theta) \int_{-\infty}^\infty \frac{ \log \|(I-\Delta)^{\frac{it}{2}}\|_{BMO(\mathbb{R}^n)\rightarrow BMO(\mathbb{R}^n)} }{\cosh(\pi t)+\cos(\pi \theta)}\;dt.
\end{align*}
The claimed inequality \eqref{intermediate_inequality} follows by manipulation of the preceding inequality and taking the supremum over $h \in L^{p'}(\mathbb{R}^n)$.

From the inequality \eqref{intermediate_inequality}, the pointwise convergence \eqref{sharp-maximal-attained}, and Fatou's lemma we deduce
\begin{align*}
    \|M^{\#}(T_{\frac{n+\theta}{2}}f)\|_{L^p(\mathbb{R}^n)} \leq C\|T_{ n/2}f\|_{L^2(\mathbb{R}^n)}^{1-\theta} \|f\|_{L^1(\mathbb{R}^n)}^{\theta}, \quad 2< p<\infty , \,\,\theta= 1-\frac{2}{p}
\end{align*}
for all simple functions $f$.  The known lower bound for the Fefferman-Stein maximal function (see e.g. \cite[Theorem 2 on p.~148]{Stein} or \cite{Seeger-Linear} for an argument which gives linear dependence in $p$ for even the estimate between the $L^p$-norms of the Hardy-Littlewood and Fefferman-Stein maximal functions) thus yields
\begin{align*}
    \|T_{ b}f\|_{L^p(\mathbb{R}^n)} \leq C\|T_{ n/2}f\|_{L^2(\mathbb{R}^n)}^{1-\theta} \|f\|_{L^1(\mathbb{R}^n)}^{\theta} \quad \text{for} \quad 2< p<\infty,
\end{align*}
for all simple functions $f$.  Finally, if $a \in C^\infty_c(\mathbb{R}^n)$ is a $\beta$-atom, we may find a sequence of simple functions $\{f_k\}$ such that $f_k\to a$ pointwise, in $L^1(\mathbb{R}^n)$, in $L^2(\mathbb{R}^n)$ and such that $\widehat{f}_k\to \widehat{a}$ in $L^2(\mathbb{R}^n)$.  By Plancherel's identity
\begin{align*}
\lim_{k \to \infty} \|T_{ n/2}f_k\|_{L^2(\mathbb{R}^n)} &= \lim_{k \to \infty} \int_{\mathbb{R}^n} (1+4\pi^2|\xi|^2)^{-n/2} |\widehat{f}_k(\xi)|^2\;d\xi \\
&=   \int_{\mathbb{R}^n} (1+4\pi^2|\xi|^2)^{-n/2} |\widehat{a}(\xi)|^2\;d\xi \\
&=  \|T_{ n/2}a\|_{L^2(\mathbb{R}^n)},
\end{align*}
which in combination with the preceding bounds for $f_k$ yield
\begin{align*}
\limsup_{k \to \infty} \|T_{ b}f_k\|_{L^p(\mathbb{R}^n)} \leq C\|T_{ n/2}a\|_{L^2(\mathbb{R}^n)}^{1-\theta} \|a\|_{L^1(\mathbb{R}^n)}^{\theta}.
\end{align*}
This inequality, the embedding $DS_\beta(\mathbb{R}^n) \hookrightarrow L^1(\mathbb{R}^n)$, and Proposition \ref{prop-L2-gamma-1} together imply
\begin{align*}
\limsup_{k \to \infty} \|T_{ b}f_k\|_{L^p(\mathbb{R}^n)} \leq C.
\end{align*}
Finally, $a \in C^\infty_c(\mathbb{R}^n)$ implies that one has pointwise convergence
\begin{align*}
\lim_{k \to \infty} T_{ b}f_k = T_ba,
\end{align*}
so that Fatou's lemma gives the claimed inequality.  This completes the proof of the theorem.
\end{proof}

\section{The estimate for \texorpdfstring{$1\leq p <2$}{p less 2}}\label{p_small_than_2}
In this section, we will give the proofs of Theorems \ref{Thm-p-less-2} and \ref{Thm-p-less-2-2}.

 \begin{lemma}\label{LP-arguments}
     Let $n\geq 2$ and $p\in [1, 2]$.  Let $a$ be a $\beta$-atom adapted to a cube $Q$. Then for $\beta \in (n-1, n]$, there exists a constant $C=C(n,p,\beta)$ such that 
    \begin{align*}
        \|\widetilde{T}_{b_p}(a)\|_{L^p(\mathbb{R}^n)} \leq C \ell(Q)^{\left( \frac{1-n}{p'-(n-\beta)}\right)}.
    \end{align*} 
 \end{lemma}
 \begin{proof}
 Let $n\geq 2$ and $p\in [1, 2]$.  Let $\beta \in (n-1,n]$ and $a$ be a $\beta$-atom adapted to a cube $Q$.
 
By \eqref{partition of unity}, we can write 
 $$\widetilde{T}_{b_p}=\sum_{j\geq 1} \widetilde{T}_{b_p}^j$$
 where $\widetilde{T}_{b_p}^j$ is a convolution operator with kernel $K_b^j$ defined in \eqref{def-Tj}.
 
     First, we decompose
\begin{align*}
    \|\widetilde{T}_{b_p}(a)\|_{L^p(\mathbb{R}^n)} &\leq \sum_{j\in \mathbb{Z}} \|\widetilde{T}_{b_p}^{j}(a)\|_{L^p(\mathbb{R}^n)} \\
    &\leq  \sum_{j\leq k} \|\widetilde{T}_{b_p}^{j}(a)\|_{L^p(\mathbb{R}^n)} + \sum_{j>k} \|\widetilde{T}_{b_p}^{j}(a)\|_{L^p(\mathbb{R}^n)}\\
    &= \mathcal{I}+ \mathcal{II},
\end{align*}
where $k$ is a positive number that we choose later.

Using the cancellation condition of the atom and Lemma \ref{kerne--esti-1}, we have
\begin{align}\label{L1-esti-grad}
    \nonumber\|\widetilde{T}_{b_p}^{j}(a)\|_{L^1(\mathbb{R}^n)}& = \int_{\mathbb{R}^n}\left| \int K_{b_p}^j(x-y)a(y)\,dy-K_{b_p}(x)\int a(y)\, dy\right| \, dx\\
    &=\int_{\mathbb{R}^n}\left| \int \left(K_{b_p}^j(x-y)-K_{b_p}(x)\right)a(y)\,dy\right| \, dx\\
    \nonumber& \leq  \ell(Q) \|\nabla K_{ b}^j\|_{L^1(\mathbb{R}^n)} \|a\|_{L^1(\mathbb{R}^n)}\\
    \nonumber & \lesssim 2^{j/p} \ell(Q).
\end{align}

Another application of Lemma \ref{kerne--esti-1} gives 
\begin{align}\label{Linfinity-esti}
  \|\widetilde{T}_{b_p}^{j}(a)\|_{L^\infty(\mathbb{R}^n)}  \leq \|K_{b_p}^{j}\|_{L^\infty(\mathbb{R}^n)}\|a\|_{L^1(\mathbb{R}^n)} \lesssim 2^{j/p}.
\end{align}

The estimates \eqref{L1-esti-grad} and \eqref{Linfinity-esti} together imply
\begin{align}\label{I}
   \nonumber \mathcal{I}= \sum_{j\leq k} \|\widetilde{T}_{b_p}^{j}(a)\|_{L^p(\mathbb{R}^n)} & \leq  \sum_{j\leq k} \|\widetilde{T}_{b_p}^{j}(a)\|_{L^1(\mathbb{R}^n)}^{1/p} \|\widetilde{T}_{b_p}^{j}(a)\|_{L^\infty(\mathbb{R}^n)}^{1-\frac{1}{p}}\\
 & \lesssim  \sum_{j\leq k} \left(\ell(Q) \|\nabla K_{ b}^j\|_{L^1(\mathbb{R}^n)}\right)^{1/p} \|K_{b_p}^{j}\|_{L^\infty(\mathbb{R}^n)}^{1-\frac{1}{p}}\\  
    & \lesssim  \sum_{j\leq k} 2^{\frac{j}{p^2}} \ell(Q)^{\frac{1}{p}} 2^{\frac{j}{p}(1-\frac{1}{p})}\\
    \nonumber & = \ell(Q)^{\frac{1}{p}} \sum_{j\leq k} 2^{\frac{j}{p}}\\
   \nonumber & =  \ell(Q)^{\frac{1}{p}} 2^{\frac{k}{p}}.
\end{align}

We next estimate $\mathcal{II}$. Note that the estimate (5.9) of \cite{Spector-Stolyarov} implies 
\begin{align}\label{L-P-esti}
    \|\widetilde{P}_ja\|_{L^{\infty}(\mathbb{R}^n)} \lesssim \frac{2^{j(n-\beta)}}{\ell(Q)^{\beta}}
\end{align}
for all $j\in \mathbb{Z}$.

By Lemma \ref{kerne--esti-1} and \eqref{L-P-esti}, introducing a second Littlewood-Paley projector as in \eqref{second_projector}, we obtain
\begin{align}\label{Linfty-est-II}
   \nonumber \|\widetilde{T}_{b_p}^{j}(a)\|_{L^\infty(\mathbb{R}^n)} & = \|K_{b_p}^{j}\ast \widetilde{P}_j a\|_{L^\infty(\mathbb{R}^n)}\\
    & \leq  \|K_{b_p}^{j}\|_{L^1(\mathbb{R}^n)} \|\widetilde{P}_j a\|_{L^\infty(\mathbb{R}^n)}
\end{align}

Young's inequality gives the estimate
\begin{align}\label{L1-esti-II}
    \|\widetilde{T}_{b_p}^{j}(a)\|_{L^1(\mathbb{R}^n)} = \|K_{b_p}^{j}\ast a\|_{L^1(\mathbb{R}^n)}
    \leq  \|K_{b_p}^{j}\|_{L^1(\mathbb{R}^n)} \|a\|_{L^1(\mathbb{R}^n)}.
\end{align}

The estimates \eqref{L1-esti-II}, \eqref{Linfty-est-II}, and Lemma \ref{kerne--esti-1} together imply
\begin{align}\label{II}
   \nonumber \mathcal{II} &\lesssim  \sum_{j>k} \|K_{b_p}^{j}\|_{L^1(\mathbb{R}^n)} \|\widetilde{P}_j a\|_{L^\infty(\mathbb{R}^n)}^{1-1/p}\\
  \nonumber &\lesssim \sum_{j>k}  2^{-j(1-\frac{1}{p})} \left[ \frac{2^{-j(\beta-n)}}{\ell(Q)^{\beta}}\right]^{1-\frac{1}{p}}\\
   & =  \frac{1}{\ell(Q)^{\beta(1-\frac{1}{p})}} \sum_{j>k} 2^{-j(1-\frac{1}{p})(1-n+\beta)}\\
   \nonumber & =  \frac{1}{\ell(Q)^{\beta(1-\frac{1}{p})}}  2^{-k(1-\frac{1}{p})(1-n+\beta)}.
\end{align}
In the last line, we have used the fact that $\beta>n-1$.

Combining the estimates \eqref{I} and \eqref{II}, we find
\begin{align}\label{Lp-esti}
    \|\widetilde{T}_{b_p}(a)\|_{L^p(\mathbb{R}^n)}\lesssim  \ell(Q)^{\frac{1}{p}} 2^{\frac{k}{p}}+ \frac{1}{\ell(Q)^{\beta(1-\frac{1}{p})}}  2^{-k(1-\frac{1}{p})(1-n+\beta)}.
\end{align}
A choice of $k$ such that 
$$2^k \approx \ell(Q)^{\left( \frac{1-n}{p'-(n-\beta)} -1\right)}$$
in the inequality \eqref{Lp-esti} yields
\begin{align*}
  \|\widetilde{T}_{b_p}(a)\|_{L^p(\mathbb{R}^n)}\lesssim \ell(Q)^{\left( \frac{1-n}{p'-(n-\beta)}\right)}. 
\end{align*}
This completes the proof of the lemma.
 \end{proof}

 \subsection{ Dimension \texorpdfstring{$n\geq 3$}{}}

 \begin{lemma}\label{lemma-small-cubes}
     Let $n\geq 3$ and $p\in [1, 2]$. Let $a$ be a $\beta$-atom adapted to a cube $Q$ with a center at zero and length $\ell(Q)<\frac{1}{3\sqrt{n}}$. Then for $\beta > \frac{n-1}{2}+\frac{1}{p}$ there exists a constant $C=C(n,p,\beta)$ independent of $a$ such that 
    \begin{align*}
        \|\widetilde{T}_{b_p}(a)\|_{L^p(\mathbb{R}^n)}  \leq C.
    \end{align*} 
 \end{lemma}
 \begin{proof}
 Let $n\geq 3$, $p\in [1, 2]$ and $\frac{n-1}{2}+\frac{1}{p} < \beta \leq n$. Let $a$ be a $\beta$-atom adapted to a cube $Q$ with a center at zero and a length $\ell(Q)<\frac{1}{3\sqrt{n}}$.
     
By the cancellation condition of the atom and Proposition \ref{wave-ker-esti}, we have
\begin{align}\label{Ker-esti-local-atom}
    \nonumber \left|\widetilde{T}_{b_p}(a)(x) \right|& =\left|\int_{Q} K_{b_p}(x-y)a(y) \, dy \right|\\
    \nonumber & =\left|\int_{Q} \left( K_{b_p}(x-y)-K_{b_p}(x)\right) a(y) \, dy \right|\\
   \nonumber &\leq  \int_{Q}\int_{0}^{1} \left| \nabla K_{b_p}(x-ty)\right| |y| |a(y)| \, dt \, dy \\
   \nonumber  & \leq  \ell(Q) \sup_{y\in Q} \left| \nabla K_{b_p}(x-y)\right|\\
   & \lesssim  \left \{
   \begin{array}{ll}
     \ell(Q)~ |1-|x||^{-1/p-1}   &  \text{if} \, \, ||x|-1|\geq 2\sqrt{n}\ell(Q)\\
      \ell(Q) ~|x|^{-N}  & \text{if} \, \, |x|\geq 2\sqrt{n}\ell(Q)
   \end{array}
   \right.
\end{align}
for all $N>0$.

We write
\begin{align*}
    \int \left| \widetilde{T}_{b_p}(a)(x) \right|^p \, dx & =  \sum_{i=1}^{4}\int_{A_i} \left| \widetilde{T}_{b_p}(a)(x) \right|^p \, dx \\
    & = J_1 + J_2 + J_3 + J_4
\end{align*}
where
\begin{align*}
    A_1& =\{x\in \mathbb{R}^n: |x|\leq 1-2\sqrt{n}\ell(Q)\},\\
    A_2& = \{x\in \mathbb{R}^n: 1-2\sqrt{n}\ell(Q) \leq |x|\leq 1+2\sqrt{n}\ell(Q)\},\\
    A_3& = \{x\in \mathbb{R}^n: 1+2\sqrt{n}\ell(Q) \leq |x|\leq 2\sqrt{n}\},\\
    A_4& =  R^{n}\setminus \bigcup_{i=1}^3 A_i.
\end{align*}

We estimate $J_1$ and $J_3$ first. The estimate \eqref{Ker-esti-local-atom} yields 
\begin{align*}
    J_1 \lesssim \ell(Q)^p \int_{|x|\leq 1-2\sqrt{n}\ell(Q)} (1-|x|)^{-1-p} \, dx\lesssim \ell(Q)^p \ell(Q)^{-p} \lesssim 1.
\end{align*}
and 
\begin{align*}
    J_3 \lesssim \ell(Q)^p \int_{ 1+2\sqrt{n}\ell(Q)\leq |x|} (|x|-1)^{-1-p} \, dx\lesssim \ell(Q)^p \ell(Q)^{-p} \lesssim 1.
\end{align*}

For $J_4$, again using  \eqref{Ker-esti-local-atom}, we obtain
\begin{align*}
    J_4 \lesssim \ell(Q)^p \int_{ 2\sqrt{n}\leq |x|} |x|^{-Np} \, dx\lesssim \ell(Q)^p \lesssim 1
\end{align*}
for $N>n/p$.

Finally, we calculate the remaining part $J_2$. Using Holder's inequality and $\ell(Q)\lesssim 1$, we have
\begin{align}\label{J2-esti}
   \nonumber J_2 & \lesssim  \ell(Q)^{\frac{1}{p}-\frac{1}{2}} \|\widetilde{T}_{b_p}(a)\|_{L^2(\mathbb{R}^n)}\\
    & \lesssim  \ell(Q)^{\frac{1}{p}-\frac{1}{2}}  \left( \int (2\pi |\xi|)^{-2b_p}|\widehat{a}(\xi)|^2 \, d\xi \right)^{1/2}\\
   \nonumber & = \ell(Q)^{\frac{1}{p}-\frac{1}{2}} \|I_{b_p}(a)\|_{L^2(\mathbb{R}^n)}
\end{align}
where $I_{b_p}$ is the Riesz potential.  Lemma \ref{Riesz-potential-esti} gives the bound
\begin{align}\label{RP-esti}
 \|I_{b_p}(a)\|_{L^2(\mathbb{R}^n)}  \lesssim   \ell(Q)^{\frac{1}{2}-\frac{1}{p}},
\end{align}
In particular, the estimates \eqref{J2-esti} and \eqref{RP-esti} together yield the bound
$$J_2\lesssim 1.$$
Finally, combining the estimates of $J_1$, $J_2$, $J_3$ and $J_4$, we obtain the desired bound. This completes the proof of the lemma.
 \end{proof}

\begin{proof}[Proof of Theorem \ref{Thm-p-less-2}]
Let $n\geq3$, $\beta \in (n-1,n]$, and $p\in [1,2]$. As argued in the proof of Theorem \ref{Thm-p-bigger-2}, it is enough to prove Theorem \ref{Thm-p-less-2} for $b=b_p$. Note that using the decomposition \eqref{Decomposition-operator} , we can write
\begin{align*}
T_{b_p}\mu = \overline{T}_{b_p}\mu+ \nu_{b_p}\ast \widetilde{T}_{b_p}\mu
\end{align*}
where $\overline{T}_{b_p}$, $\widetilde{T}_{b_p}$ are Fourier multipliers operators defined in \eqref{goodkernel} and \eqref{main-operator} and $\nu_{b_p}$ is bounded measure. Since operator $\overline{T}_{b_p}$ maps from $L^1(\mathbb{R}^n)$ to $L^p(\mathbb{R}^n)$ for any $1\leq p \leq +\infty$ and $\nu_{b_p}$ is bounded measure, it suffices to prove the estimate for $\widetilde{T}_{b_p}$.

Let $a \in C^\infty_c(\mathbb{R}^n)$ be an $\beta$-atom. Again by Proposition \ref{smooth_atoms} it is enough to prove that there exists a uniform constant $C>0$ independent of $a$ such that 
\begin{align*}
    \|\widetilde{T}_{b_p}(a)\|_{L^p(\mathbb{R}^n)}\leq C.
\end{align*}
Since the operator $\widetilde{T}_{b_p}$ is translation invariant, without loss of generality, let us assume that the support of $a$ is contained in a cube $Q$ with center at the origin and side length $\ell(Q)$.

First, let us consider the case for $\ell(Q)>\frac{1}{3\sqrt{n}}$. Note that Lemma \ref{LP-arguments} gives
\begin{align*}
  \|\widetilde{T}_{b_p}(a)\|_{L^p(\mathbb{R}^n)}\lesssim \ell(Q)^{\ -\frac{n-1}{p'-(n-\beta)}} \lesssim 1. 
\end{align*}
In the above estimate, we used the fact that $n-1<\beta\leq n$ and $\ell(Q)>\frac{1}{3\sqrt{n}}$. This completes the proof for large cubes.  We next consider the case $\ell(Q)\leq \frac{1}{3\sqrt{n}}$. For $n\geq 3$ and $p\in [1,2]$, we have $n-1 > \frac{n-1}{2}+\frac{1}{p}$. Therefore, the estimate for this case directly follows from Lemma \ref{lemma-small-cubes}.  This completes the proof of Theorem \ref{Thm-p-less-2}.
\end{proof}

\subsection{Dimension \texorpdfstring{$n=2$}{}}

\begin{lemma}\label{lemma-small-cubes-2}
     Let $n=2$ and $p\in (1, 2]$. Let $a$ be a $\beta$-atom adapted to a cube $Q$ with center at zero and length $\ell(Q)<\frac{1}{3\sqrt{2}}$. Then for $\beta > \frac{2}{p}$ there exists a constant $C=C(\beta,p)$ independent of $a$ such that 
    \begin{align*}
        \|\widetilde{T}_{b_p}(a)\|_{L^p(\mathbb{R}^2)}  \leq C.
    \end{align*} 
 \end{lemma}
 \begin{proof}
     The proof of Lemma \ref{lemma-small-cubes-2} is identical to that of the Lemma \ref{lemma-small-cubes}, except in the estimate for $J_2$, we use Lemma \ref{Riesz-potential-esti-n=2} instead of Lemma \ref{Riesz-potential-esti}, which gives the restriction $\beta>\frac{2}{p}$.
 \end{proof}

 \begin{proof}[Proof of Theorem \ref{Thm-p-less-2-2}]
  The argument is similar to that in Theorem \ref{Thm-p-less-2-2}, where again for large atoms the use of Lemma \ref{LP-arguments}, and for small atoms, Lemma \ref{lemma-small-cubes-2} (in place of Lemma \ref{lemma-small-cubes}).  We therefore omit the details for brevity.  This completes the proof of Theorem \ref{Thm-p-less-2-2}.
 \end{proof}

\section*{Acknowledgments} 
The authors would like to thank Po Lam Yung for discussions regarding complex interpolation with $BMO$ as an endpoint, in particular for communicating to us the construction used in the proof of Lemma \ref{pointwise_approximation}, as well as the arguments involving analyticity in the proof of Theorem \ref{Thm-p-bigger-2}. 
The authors would like to thank Andreas Seeger for discussions regarding the dependence of the constant in the Fefferman-Stein maximal inequality.
Needless to say that we remain responsible for any remaining shortcomings.

R.~Basak is supported by the National Science and Technology Council of Taiwan under research grant number 113-2811-M-003-014.  D. Spector is supported by the National Science and Technology Council of Taiwan under research grant number 113-2115-M-003-017-MY3 and the Taiwan Ministry of Education under the Yushan Fellow Program.

\end{document}